\documentclass[10pt,reqno]{amsart}
\usepackage{amsmath}
\usepackage{amssymb}
\usepackage{amsthm}
\usepackage[usenames]{color}
\usepackage{eepic,epic}

\newtheorem{thm}{Theorem}[section]
\newtheorem{cor}[thm]{Corollary}
\newtheorem{lem}[thm]{Lemma}
\newtheorem{prop}[thm]{Proposition}
\theoremstyle{definition}

\theoremstyle{remark}
\newtheorem{rem}[thm]{Remark}
\numberwithin{equation}{section}

\newcommand{\dist}{\text{dist }}

\begin{document}

\title{Maximal functions for multipliers on compact manifolds}

\author{Woocheol Choi}
\address[Woocheol Choi]{Department of Mathematical Sciences, Seoul National University, 1 Gwanak-ro, Gwanak-gu, Seoul 151-747, Republic of Korea}
\email{chwc1987@math.snu.ac.kr}

\subjclass[2000]{Primary}

\maketitle

\begin{abstract}
Let $P$ be a self-adjoint positive elliptic (-pseudo) differential operator on a compact manifold $M$ without boundary. For  a function $m \in L^{\infty}[0,\infty)$ satisfying a H\"ormander-Mikhlin type condition, Seeger and Sogge \cite{ss} proved that the multiplier theorem $\|m(P) f \|_{L^p(M)} \leq C_p \| f\|_{L^p (M)}$ holds. In this paper, we prove that $\| \sup_{1 \leq i \leq N} |m_i (P) f |\|_{L^p} \leq C_p (\log (N+1))^{1/2} \| f\|_{L^p} $ holds when $\{ m_i\}_{i=1}^{N}$ uniformly satisfy the condition. This result is sharp when $M$ is  $n$ dimensional torus.
\end{abstract}

\section{introduction}

In this paper we study the multiplier operators on compact manfiolds without boundary. The $L^p$-boundedness property of multipliers was established by Seeger and Sogge \cite{ss} under the Hormander-Mikhlin type condition. We obtain a result on $L^p$-boundedness of maximal functions of the multipliers.
\

Let $M$ be a compact manifold of dimension $n \geq 2$ without boundary. We consider a first order elliptic pseudo-differential operator $P$. We assume that $P$ is positive and self-adjoint with respect to a $C^{\infty}$ density $dx$ on $M$. Since the inverse operator of $P$ is compact on $L^2(M) = L^2 (M,dx)$ and $P$ is self-adjoint, the spectral theorem implies that
\begin{eqnarray*}
L^2 (M) = \sum_{j=1}^{\infty} E_j,
\end{eqnarray*}
where $E_j$ is an eigenspace of dimension one of the operator $P$ with an eigenvalue $\lambda_j$. Here we assume that $\{ \lambda_j\}$ is arranged as $0 < \lambda_1 \leq \lambda_2 \leq \cdots .$ Let $e_j$ be the projection operator onto the eigenspace $E_j$. Then for any $f \in L^{2}(M)$ we have 
\begin{eqnarray*}
f = \sum_{j=1}^{\infty} e_j (f),
\end{eqnarray*}
and
\begin{equation}\label{eq-f-l2} 
\| f\|_{L^2(M)}^2 = \sum_j^{\infty} \|e_j (f)\|_{L^2(M)}^2.
\end{equation}
For $m \in L^{\infty}([0,\infty))$ the multiplier operator $m(P): L^2 (M) \rightarrow L^2 (M)$ associated to $P$ is defined by
\begin{equation}\label{mp}
m(P) f = \sum_{j=1}^{\infty} m(\lambda_j) e_j (f),\quad f \in L^2 (M).
\end{equation}
From \eqref{eq-f-l2} we see that $m(P)$ is bounded on $L^2 (M)$ for any $m \in L^{\infty}([0,\infty))$. Meanwhile, we need to impose an additional condition on $m$ to guarantee that $m(P)$ is bounded on $L^p (M)$ for $p \neq 2$. Under a condition on $m$ involving that $m$ is a $C^{\infty}$ function, the $L^p$-bound of $m (P)$ for $1<p< \infty$ was a classical result (see \cite{taylor}). A sharp result was obtained later by Seeger and Sogge \cite{ss}  where they established the $L^p$-bound result for $1<p<\infty$ under the H\"ormander-Mikhlin type condition. To state the result, we take a function $\beta \in C_0^{\infty} ((1/2,2))$ such that $\sum_{-\infty}^{\infty} \beta(2^j s) = 1 , s>0$, and introduce the functional
\begin{eqnarray}\label{smooth}
[m]_s = \sup_{0 \leq \alpha \leq s}\left[\sup_{\lambda >0} \lambda^{-1} \int^{\infty}_{-\infty} | \lambda^{\alpha} D_s^{\alpha} ( \beta(s/\lambda) m(s))|^2 ds \right].
\end{eqnarray}
We state the multiplier theorem of Seeger and Sooger \cite{ss}:
\begin{thm}[\cite{ss}] Let $s \in \mathbb{R}^{+}$ such that $s>\frac{n}{2}$. Then for any $m \in L^{\infty}([0,\infty))$ with finite $[m]_s$, we have
\begin{equation}\label{eq-main}
\| m(P) f \|_{L^p (M)} \leq C \left[m\right]_s \| f\|_p, \qquad 1 < p < \infty,\quad \forall f \in L^p (M).
\end{equation} 
Here the constant $C$ is independent of $m$ and $f$. 
\end{thm}
The aim of this paper is to obtain a $L^p$ bound of the maximal functions of any $N$-multipliers with $N \in \mathbb{N}$. This is the main result.
\begin{thm}\label{main}
For $r>0$ let  $s >\frac{n}{r}$. Then for each $p \in (r,\infty)$  we have 
\begin{eqnarray*}
\| \sup_{1 \leq i \leq N} |m_i (P) f | \|_{L^p (M)} \leq C_p \sup_{1\leq i \leq N} [m_i]_s  \cdot (\log (N+1))^{1/2}  \| f\|_p,\quad \forall f \in L^p (M).
\end{eqnarray*}
Here the constant $C_p$ is independent of $N$.
\end{thm}
For studying the mulitplier $m(P)$ on compact manifold it is standard to divide the multiplier $m(P)$ into two parts by using the method combining the Schr\"odinger propergator $e^{itP}$. We shall handle each part separately. In local coordinates, the first part will be studied with studying properties of the kernels. On the other hand we shall bound the second part in $L^{\infty}$ space using a $L^p-L^q$ estimate of the spectral projection operators. 
\

The study of this problem was motivated by the result of Grafakos-Honzik-Seeger \cite{GHS} where for the maximal functions of $N$ multipliers on the Euclidean space, they obtained the $L^p$ bound with the constant $(\log (N+1))^{1/2}$ as in \eqref{eq-main}.
This growth rate of $N$ is known to be sharp due to the example which was constructed in \cite{CG}.
\


The rest of this paper is organized in the following way. In section 2 we divide the multipliers into a major part and a remainder part. In addition each part will be decomposed further using a dyadic cut-off functions. In section 3 we obtain the desired estimate for the remainder part first. The In Section 4 we shall further decompose the main part into a local operator and a remainder term which will be shown to be small enough. Section 5 is devoted to study the kernels of the local operator. Based on this we shall prove Theorem \ref{main} in Section 6.

\section{Preliminaries}
In this section we review some basic results on the spectral decomposition associated to a self-adjoint elliptic operator on a compact manifold, and the definition of the multiplier operators. Next we review the expression of the multipliers using the Schr\"odinger propagator and a  bound property of the spectral projection operators. We refer to the book \cite{sogge} for more details. In the last part, we will decompose the multipliers in two parts which will be treated separately in the proof of the main theorem. 
\

Let $M$ be a compact manifold with a density $dx$ and $P$ be a first-order self-adjoint positive elliptic operator on $L^2 (M, dx)$. Then, by spectral theory, the oprator $P$ has positive eigenvalues $\lambda_1 \leq \lambda_2 \leq \cdots $ associated to orthonormal eigenfunctions $e_1, e_2,\cdots.$ 
By the orthonormality we have 
\begin{eqnarray}\label{ortho}
\int_M e_j(x) e_i (x)dx =\delta_{ij}.
\end{eqnarray}
Let $E_j : L^2 \rightarrow L^2$  be the projection maps onto the one-dimensional eigenspace $\varepsilon_j$ spanned by $e_j$. Then we have  $P = \sum_{j=1}^{\infty} \lambda_j E_j$ and
\begin{eqnarray}\label{eq-ej}
E_j f (x) = e_j (x) \int_M f(y) \overline{e_j(y)} dy.
\end{eqnarray}
For a function $m \in L^{\infty}([0,\infty))$ we define the multiplier $m(P): L^2 (M) \rightarrow L^2 (M)$ in the following way
\begin{equation}\label{eq-mpf}
m(P) f := \sum_{j=1}^{\infty} m(\lambda_j ) E_j (f)= \sum_{j=1}^{\infty} m (\lambda_j ) \left( \int_M f(y) e_j (y) dy \right) e_j (x).
\end{equation}
Let $\mathcal{K}_m \in D(M \times M )$ be the kernel of $m(P)$. From the above we see that
\[\mathcal{K}_m (x,y) = \sum_{j=1}^{\infty} m( \lambda_j ) e_j (x) e_j (y).\]
Next we recall the expression using the Schrodinger propergator $e^{itP}$;
\begin{equation}\label{eq-m-fourier}
m(P) = \int^{\infty}_{-\infty} e^{it P} \hat{m} (t) f dt.
\end{equation}
We have the following result on the operator $e^{itP}$. 
\begin{thm}[{\cite[Theorem 3.2.1]{sogge}}]\label{thm-para} Let $M$ be a compact $C^{\infty}$ manifold and let $P \in {\phi}_{cl}^{1}(M)$ be elliptic and self-adjoint with respect to a positive $C^{\infty}$ density $dx$. Then there is an $\epsilon >0$ such that when $|t| < \epsilon$, 
\begin{equation}\label{eq-eitp}
e^{it P} = Q(t) + R(t)
\end{equation}
where the remainder has kernel $R(t,x,y) \in C^{\infty} ( [-\epsilon, \epsilon] \times M \times M) $ and the kernel $Q(t,x,y)$ is supported in a small neighborhood of the diagonal in $M \times M$. Furthermore, suppose that local coordinate are chosen in a patch $\Omega \subset M$ so that $dx$ agrees with Lebesque measure in the corresponding open subset $\widetilde{\Omega} \subset \mathbb{R}^n$; then, if $\omega \subset \Omega $ is relatively compact, $Q(t,x,y)$ takes the following form when $(t,x,y) \in [-\epsilon, \epsilon] \times M \times \omega$.
\begin{eqnarray*}
Q(t,x,y) = (2\pi)^{-n} \int e^{i [\phi(x,y,\xi) + t p(y,\xi)]} q (t,x,y,\xi) d \xi.
\end{eqnarray*}
\end{thm}
As given in the above theorem, we shall rely heavily on the precise local formula of $e^{itP}$ for small $|t|$. On the other hand,  to handle the part $e^{itP}$ for large $t$, we shall use the $L^p - L^q$ bound of spectral projection operators. Hence we shall decompose $m(P)$ into two parts according the values of $t$. In addition, we also decompose the multipliers into the dyadic pieces. 

For the decomposition we take functions $\phi_0 \in C_{0}^{\infty}([0,1))$ and  $\phi \in C_0^{\infty} (1/4,1)$ such that  $\sum_{j=0}^{\infty} \phi_j^3 (s) = 1 $ for all $ s \geq 0$ where $\phi_j(s) := \phi(s/2^j)$ for $j \geq 1$. For given $m \in L^{\infty}([0,\infty))$ we set $m_j (\cdot):= m(\cdot) \phi_j (\cdot)$. Then we have
\begin{eqnarray}\label{eq-mpf}
m(P) f = \sum_{j=1}^{\infty}\phi_j (P) \left[ m(P)\phi_j (P) \right] \phi_j (P) f =\sum_{j=1}^{\infty}\phi_j (P) m_j (P) \phi_j (P) f.
\end{eqnarray}
Using  \eqref{eq-m-fourier} we write
\begin{eqnarray}\label{mj}
m_j (P) = \int e^{itP} \widehat{m_j} (t) dt.
\end{eqnarray}
Let us take a function $\rho \in S(\mathbb{R})$ satisfying $\rho(t)=1, |t| \leq \frac{\epsilon}{2}$ and $\rho(t) =0, |t| > \epsilon$, and we split the integration \eqref{mj} as follows
\begin{equation}\label{eq-ajp}
m_j (P) := A_j (m, P) + R_j (m, P),
\end{equation}
where
\begin{equation}\label{eq-ajmp}
A_j (m, P) = \int e^{itP} \widehat{m_j}(t) \rho(t) dt \quad \textrm{and}\quad R_j (m, P) = \int e^{itP} \widehat{m_j} (t) (1-\rho(t)) dt.
\end{equation}
Next, we want to express $m_j (P)$ further in a composition form with an aim to achieve a $L^p$ bound for $p>2$ and a cancellation property of kernels (see Lemma \ref{lem-hj} and Corollary \ref{lem-hj-2}).
 For this we take a $C^{\infty}$ function $\widetilde{\phi}$ supported on $(\frac{1}{8}, 2)$ such that $\widetilde{\phi} = 1 $ on $(\frac{1}{4}, 1)$. Letting $\widetilde{\phi}_j (\cdot) = \widetilde{\phi} ( \frac{\cdot}{2^j})$, we have $\widetilde{\phi}_j \cdot \phi_j = \phi_j$, and so it holds that
\begin{equation}\label{eq-m-split-1}
m_j (P) = m_j (P) \widetilde{\phi}_j (P)= A_j (m, P)\circ \widetilde{\phi}_j (P) + R_j (m, P) \circ \widetilde{\phi}_j (P).
\end{equation}
Injecting this into \eqref{eq-mpf} we have $m(P) = A(m,P) + R(m,P)$, where
\begin{equation}
A (m,P)f := \sum_{j=1}^{\infty} \phi_j (P) \left[ A_j (m,P) \circ \widetilde{\phi}_j (P)\right] \phi_j (P) f
\end{equation}
and
\begin{equation}\label{eq-rmp}R (m,P) f := \sum_{j=1}^{\infty} \phi_j (P)\left[ R_j (m,P)\circ \widetilde{\phi}_j (P)\right] \phi_j (P) f.
\end{equation}
We shall study the maximal function of $R(m,P)$ and it of $A(m,P)$ in different ways. First we shall obtain the following result.
\begin{prop}\label{prop-rem}
We have
\begin{equation}
\left\| \sup_{1 \leq i \leq N} \left|R(m_i, P) f\right|\right\|_{L^p (M)} \leq C \| f\|_{L^p (M)}.
\end{equation}
\end{prop}
This result will be proved in Section 3. For the main part, we shall prove the following result in the subsequent sections.
\begin{prop}\label{prop-main}
For $r>0$ let $s>\frac{n}{r}$. Then for each $p \in (0,\infty)$ we have
\begin{equation}
\left\| \sup_{1 \leq i \leq N} \left| A(m_i, P) f \right| \right\|_{L^p (M)} \leq 
C_p \sup_{1\leq i \leq N} [m_i]_s  \cdot (\log (N+1))^{1/2} \left\| f\right\|_{L^p (M)},\quad \forall f \in L^{p}(M).
\end{equation}
\end{prop}
Given these results, it is easily follows the proof of the main theorem.
\begin{proof}[Proof of Theorem \ref{main}]
Given functions $m^1,\cdots, m^N$ we write each multiplier $m^j (P)$ as $m^j (P) = A (m^j , P) + R(m^j, P)$. Then we have
\begin{equation}
\begin{split}
\left\| \sup_{1\leq i \leq N} | m^j (P) f|\right\|_{L^p (M)}& \leq \left\| \sup_{1 \leq i \leq N} \left| A(m_i, P) f \right| \right\|_{L^p (M)} + \quad\left\| \sup_{1 \leq i \leq N} \left|R(m_i, P) f\right|\right\|_{L^p (M)}
\\
& \leq \sqrt{\log (N+1)} \left\| f\right\|_{L^p (M)}.
\end{split}
\end{equation}
It completes the proof.
\end{proof}
\section{Estimates for the remainder terms}  
To handle the remainder part, we shall rely on the $L^p - L^q$ boundedenss result of the spectral projection operators
\begin{equation} 
\chi_{\lambda} f = \sum_{\lambda_j \in [\lambda, \lambda+1]} E_j f,\qquad \lambda \in [0,\infty).
\end{equation}
We recall the result from \cite{sogge}.
\begin{lem}[see {\cite[Lemma 4.2.4 and Lemma 5.1.1]{sogge}}]\label{lem-82}
Let $M$ be a compact manifold. Then there exists a constant $C>0$ such that 
\begin{equation}\label{eq-infty-2}
\| \chi_{\lambda} f \|_{L^{\infty} (M)} \leq C (1+ \lambda)^{(n-1)/2} \| f\|_{L^2 (M)},
\end{equation}
and
\begin{equation}
\| \chi_{\lambda} f \|_{L^2 (M)} \leq C (1+ \lambda)^{\frac{n}{2}-1} \| f\|_{L^1 (M)},
\end{equation}
where the constant $C$ is independent of $\lambda$.
\end{lem} 
To prove Proposition \ref{prop-rem}, we shall obtain the following $L^{\infty}$ bound.
\begin{prop}\label{remainder}
Suppose that $m \in L^{\infty}(0,\infty)$ satisfies $[m]_s < \infty$ for some $s>\frac{n}{2}$. Then there exists a constant $C= C([m]_s )$ such that $R(m,P)$ given by \eqref{eq-rmp} satisfies
\begin{eqnarray*}
\| R(m, P) f \|_{L^{\infty}} \leq C \| f\|_p.
\end{eqnarray*}
\end{prop}
\begin{proof}
By the decomposition \eqref{eq-rmp} it is sufficient to prove that
\begin{equation}\label{eq-rj-inf}
\left\| \phi_j (P) [R_j (m, P)\circ \widetilde{\phi}_j (P)] \phi_j (P) f \right\|_{L^{\infty}} \leq C  2^{j \left( \frac{n}{2} -s\right)}  \left\| f\right\|_{L^p}.
\end{equation}
Applying \eqref{eq-infty-2} we have
\begin{eqnarray}\label{eq-r-28}
\left\| \phi_j (P)[ R_j (m, P) \circ \widetilde{\phi}_j (P)]  \phi_j (P) f \right\|_{L^{\infty}}^{2} \leq C  2^{j(n-1)} \|\phi_j (P) [ R_j (m, P)\circ \widetilde{\phi}_j (P)] \phi_j (P) f \|_{L^2}^2
\end{eqnarray}
Using the fact that $|\phi_j| , |\widetilde{\phi}_j| \leq 1$ and the orthogonality, we have
\begin{equation}
\left\| \phi_j (P) \left[ R_j (m,P) \circ \widetilde{\phi}_j (P)\right] \phi_j (P) f \right\|_{L^2} \leq \left\| \left[ R_j (m,P)\right] f\right\|_{L^2}.
\end{equation}
Let $\tau_{j} (r) = [(1- \rho(t)) \hat{m}_j]^{\vee} (r)$. Then, using \eqref{eq-ajmp} we have
\begin{equation}
R_j (m,P) = \int e^{itP} \hat{\tau}_j (t) dt.
\end{equation}
Splitting the $L^2$-norm and using Lemma \ref{lem-82} we deduce that
\begin{equation}\label{eq-rjpf}
\begin{split}
\left\| \tau_j (P) f \right\|_{L^2}^2 &\leq \sum_{k=0}^{\infty} \sup_{r \in [k, k+1)} |\tau_j (r)|^2 \left\| \chi_k f \right\|_{L^2}^2
\\
&\leq \sum_{k=0}^{\infty} \sup_{ r \in [k, k+1)} |\tau_j (r)|^2  (1+k)^{n-1} \left\| f\right\|_{L^1}.
\end{split}
\end{equation}
We claim that 
\begin{equation}\label{eq-r-main}
\sum_{k \in [2^{j-2}, 2^{j+2}]} \sup_{r \in [k,k+1)} |\tau_j (r)|^2 \leq C 2^{j (n-2s)}.
\end{equation}
To show this, applying the fundamental theorem of calculus and the Casuchy Schwartz inequality, we dominate it in the following way
\begin{equation}
\begin{split}
\sum_{k \in [2^{j-2}, 2^{j+2}]} \sup_{r \in [k,k+1)} |\tau_j (r)|^2  &\leq \int |\tau_{j}(r)|^2 dr + \int |\tau'_{j}(r)|^2 dr 
\\
& = \frac{1}{2\pi} \int |\hat{m}_{j}(t) (1- \rho(t))|^2 dt 
\\
&\quad + \frac{1}{2\pi} \int |t \hat{m}_{j} (t)|^2 | ( 1- \rho(t))|^2 dt.
\end{split}
\end{equation}
Note that $\rho (t) =1$ for $|t| < \epsilon/2$, so we can bound this by
\begin{equation}
\begin{split}
& \frac{1}{2\pi} 2^{-j(1+2s)} \int |t^s \hat{m}_{j}(t/2^{j})|^2 dt 
\\
&\quad\quad = 2^{-j (1+2s)} \int |D_{r}^s (2^j m_{j}(2^j r)|^2 dr 
\\
&\quad\quad = 2^{j (1-2s)} \cdot \left\{ 2^{-j} \int |2^{-js} D_{r}^{s} (\beta( r/2^j) m (r))|^2 dr \right\}.
\end{split}
\end{equation}
By condition \eqref{smooth} of $m$ it gives the desired bound. It proves the claim \eqref{eq-r-main}.

Moreover it is easy to see that
\begin{equation}\label{eq-r--rem}
\tau_j (r) = \left[ \hat{m}_j (\cdot) (1- \rho(\cdot))\right]^{\wedge} (r) = O\left( (|r|+2^{j})^{-N}\right),
\end{equation}
for any $N\in \mathbb{N}$ if $\tau \notin \left[ 2^{j-2}, 2^{j+2}\right]$. Injecting \eqref{eq-r-main} and \eqref{eq-r--rem} into \eqref{eq-rjpf} we obtain
\begin{equation}
\left\| \tau_j (P) f \right\|_{L^2}^2 \leq 2^{j (n-2s)} \| f\|_{L^1 (M)}.
\end{equation}
Combining this with \eqref{eq-r-28} gives the estimate \eqref{eq-rj-inf}. It completes the proof.
\end{proof}
Modifying the proof of the above lemma, we can deduce the following result.
\begin{lem}\label{lem-r} Suppose that $m$ satisfies the condition \eqref{smooth}. Then we have
\begin{eqnarray*}
\| r_j (P) f \|_{L^{\infty}} \leq 2^{j \left( \frac{n}{2}-s\right)} \| f\|_{L^p},\quad 1 < p < \infty.
\end{eqnarray*}
\end{lem}
\begin{proof}
We have $\tau_j (P) f = \sum_{k=0}^{\infty} \chi_k \tau_j (P) f$ where $\chi_k$ is the spectral projection operator. Using Lemma \ref{lem-82} we deduce that
\begin{equation}
\left\| \tau_j (P) f \right\|_{L^{\infty}} \leq \sum_{k=0}^{\infty} \left\| \chi_k \tau_j (P) f \right\|_{L^{\infty}(M)} \leq \sum_{k=0}^{\infty} 2^{k(\frac{n-1}{2})} \left\| \chi_k \tau_j (P) f \right\|_{L^2 (M)}.
\end{equation}
We have
\begin{equation}
\left\| \chi_k  \tau_j (P) f\right\|_{L^2 (M)} \leq \sum_{2^k \leq m < 2^{k+1}} \sup_{m \leq t < m+1} \left| \tau_j (t)\right|^2 \left\| \chi_k f \right\|_{L^2}.
\end{equation}
For $j-2 \leq k \leq j+2$, as in \eqref{eq-r-main} we have
\begin{equation}\label{eq-sup}
\sum_{2^k \leq m < 2^{k+1}} \sup_{m \leq t < m+1} |\tau_j (t)|^2 \leq 2^{j (n-2s)}.
\end{equation}
For $|k -2^j | > 2^j$ we have $\tau_j (k) = ( \hat{m}_j (\cdot)(1- \rho(\cdot))^{\wedge} (k) = O ((|k|+ 2^j)^{-N}).$  The proof is completed. 
\end{proof}

\begin{proof}[Proof of Proposition \ref{prop-rem}]
By Proposition \ref{remainder} we have
\begin{equation}
\begin{split}
\left\| \sup_{1\leq i \leq N} |R(m_i ,P) f|\right\|_{L^p (M)} &\leq |\textrm{vol}(M)|^{1/p}\left\| \sup_{1\leq i \leq N} |R(m_i ,P) f|\right\|_{L^{\infty} (M)}
\\
&\leq |\textrm{vol}(M)|^{1/p}\sup_{1\leq i \leq N}  \biggl\| |R(m_i ,P) f|\biggr\|_{L^{\infty} (M)}
\\
&\leq C |\textrm{vol}(M)|^{1/p} \| f\|_{L^p (M)}.
\end{split}
\end{equation}
It gives the desired result.
\end{proof}

\section{Estimates for the main term}
In this section we begin to study the operator $A_j (m, P) \circ \widetilde{\phi}_j (P)$ defined in Section 2. 
We shall divide $A_j (m,P) \circ \widetilde{\phi}_j (P)$ further into a major local operator and its remainder term. We shall obtain a uniform $L^{\infty}$ bound for the remainder part. 
\

We set 
\begin{equation} 
m_j^{loc}(P) = \int Q(t) \widehat{m_j}(t) \rho(t) dt.
\end{equation}
Then we have the following result.
\begin{lem}
For any $m \in L^{\infty}([0,\infty))$, we have 
\begin{equation}\label{eq-m-tilde} 
A_j (m, P) = m_j^{loc}(P) + O (2^{-jN}),\quad \quad j \geq 1.
\end{equation}
\end{lem}
\begin{proof}
Recalling \eqref{eq-eitp} we have
\begin{eqnarray*}
A_j (m, P) = m_j^{loc}(P) + \int  R(t) \widehat{m_j}(t) \rho(t) dt.
\end{eqnarray*} 
Therefore it suffices to show
\begin{eqnarray}\label{remain}
\int R(t,x,y) \rho(t) \widehat{m_j} (t)dt = O_N (2^{-jN}) \quad \textrm {for all} ~ N \in \mathbb{N}.
\end{eqnarray}
Applying the Fourier transform we have
\begin{eqnarray*}
\int R(t,x,y) \rho(t) \widehat{m_j}(t)dt = \int {[R(\cdot, x, y)\rho(\cdot)]^{\wedge}}(t) m(t) \phi\left( \frac{t}{2^j}\right) dt. 
\end{eqnarray*}
Note that the support of $\phi(\frac{\cdot}{2^j})$ is contained in $\{ t \in \mathbb{R}^{+}| 2^{j-1} \leq t \leq 2^{j+1}\}$. In addition we have $m \in L^{\infty}(\mathbb{R})$ and $R(t,x,y) \in C^{\infty} ([-\epsilon, \epsilon] \times M \times M)$. Thus, for any given $N \in \mathbb{N}$, we have ${[R(\cdot,x,y)\rho(\cdot)]^{\wedge}}(t) m(t) \phi\left( \frac{t}{2^j}\right)= O (2^{-jN})$ for $j \in \mathbb{N}$.
Hence we have
\begin{eqnarray}\label{remain}
\int R(t,x,y) \rho(t) \widehat{m_j} (t)dt = O_N (2^{-jN}) \quad \textrm {for all} ~ N \in \mathbb{N}.
\end{eqnarray}
It shows \eqref{remain} and so the proof is completed.
\end{proof}
\

Let ${K_j}(x,y)$ be the kernel of $\int Q(t) \widehat{m_j}(t) \rho(t)dt$ for $Q(t)$ given by \eqref{eq-eitp}.
We recall the $L^2$-bound result obtained by Seeger-Sogge \cite{ss}(see also (5.3.9') in \cite{sogge}).
\begin{lem}\label{kernell2}
Suppose that $m \in L^{\infty}[0,\infty)$ satisfies the condtion \eqref{smooth} for a $s>0.$ 
 Then for $j \in \mathbb{N}$ we have ${K_j}(x,y) = 2^{nj} K_j^* (2^j x, 2^j  y)$ for some function $K_j^* \in C^{1} (M \times M)$ satisfying
\begin{eqnarray}\label{eq-k2}
\int_M |D_{y}^{\alpha} K_j^* (x,y)|^2 ( 1 + |x-y|)^{2s} dy \leq C, \quad 0 \leq |\alpha| \leq 1,
\end{eqnarray}
where the constant $C$ is independent of $j \in \mathbb{N}$ and $x \in M$.
\end{lem}
\begin{rem}Applying H\"older's inequality to \eqref{eq-k2} and a change of variables we can deduce that
\begin{equation}\label{eq-k-l1}
\int |K_j (x,y)| dy = \int |{K}_j^* (2^j x, y)| dy \leq C\quad \textrm{for}~ j \in \mathbb{N}.
\end{equation}
\end{rem}
We recall the definition 
\begin{equation*}
\psi_j^{loc}(P):= \int Q(s) \widehat{\psi}_j (s) \rho (s) ds.
\end{equation*}
Now we study the properties of the kernel of the projection operators given by a smooth bump function.
\begin{lem}\label{lem-smooth-loc}
Let $\psi \in C^{\infty}(1/2,1)$ and set $\psi_j (\cdot):= \psi(\cdot/2^j)$ for $j \in \mathbb{N}$. For any $N \in \mathbb{N}$ the operator $\psi_j (P)$ defined by \eqref{eq-m-split-1} is of the form
\begin{equation}\label{eq-zetap}
\psi_j (P) = \psi_j^{loc}(P) + O (2^{-jN}),\quad j \in \mathbb{N}.
\end{equation}
Moreover, the kernel $\mathcal{K}(\psi_j)$ of $\psi_j (P)$ satisfies uniformly for $j\in \mathbb{N}$ the estimate
\begin{equation}\label{eq-ljc}
\int_M \left| \mathcal{K}(\psi_j) (x,y) \right| dx \leq C.
\end{equation}
\end{lem}
\begin{proof}
Recalling \eqref{eq-eitp} and \eqref{eq-ajp} we have
\begin{equation}
 \psi_j (P) = \int Q(s) \widehat{\psi}_j (s) \rho (s) ds + \int R(s) \widehat{\psi}_j (s) \rho (s) ds + \int e^{itP} \widehat{\psi}_j (s) (1- \rho (s)) ds.
\end{equation}
As in \eqref{remain} we see that $\int R(s) \rho(s) \widehat{\psi}_j (s) ds = O_N (2^{-jN})$.  Next, we note that a smooth function $\psi \in C_0^{\infty}(1/8, 2)$ satisfies the condition \eqref{smooth} for any $s>0$. Therefore, for any $N \in \mathbb{N}$, we may apply Lemma \ref{lem-r} with $s=N$ to deduce that
\begin{equation}
\int e^{itP} \widehat{\psi}_j (s) (1- \rho (s)) ds = O_N (2^{-jN}).
\end{equation}
It proves the validitiy of \eqref{eq-zetap}.
\

To show \eqref{eq-ljc} we let ${\Psi}_j$ be the kernel of $\psi_j^{loc}(P)$. By \eqref{eq-k-l1} we have 
\begin{equation}
\int \left| {\Psi}_j (x,y) \right| dy  \leq C.
\end{equation}
From this and using \eqref{eq-zetap} we see that
\begin{equation}
\int \left|{\mathcal{K}} (\psi_j ) (x,y)\right| dy \leq \int \left| {\Psi}_j (x,y) \right| dy + \int O(2^{-jN}) dy \leq C,
\end{equation}
which gives \eqref{eq-ljc}. Thus the lemma is proved.
\end{proof}
\begin{rem}
We note that the functions $\phi$ and $\widetilde{\phi}$ defined in Section 2 satisfies the assumption of the above lemma. Therefore we may use the formula \eqref{eq-zetap} for $\phi$ and $\widetilde{\phi}$.
\end{rem}
We have the following result.
\begin{lem}
For $m \in L^{\infty}[0,\infty)$ we have
\begin{equation}\label{eq-ajb}
A_j (m, P) \circ \widetilde{\phi}_j  (P) =  m_j^{loc}(P) \circ \widetilde{\phi}_j^{loc}(P).+ O (2^{-jN}) \quad \forall j \in \mathbb{N}.
\end{equation}
\end{lem}
\begin{proof} Using \eqref{eq-m-tilde} we have
\begin{eqnarray*}
A_j (m, P) \circ \widetilde{\phi}_j (P) = m_j^{loc}(P)  \circ \widetilde{\phi}_j (P) + O (2^{-jN})\widetilde{\phi}_j (P).
\end{eqnarray*}
By \eqref{eq-ljc} we see  $O(2^{-jN}) \circ \widetilde{\phi}_j (P) = O(2^{-jN})$. We can also use \eqref{eq-zetap} and \eqref{eq-k-l1} to see that
\begin{equation}
m_j^{loc}(P)\circ \widetilde{\phi}_j (P) = m_j^{loc}(P) \circ \widetilde{\phi}_j^{loc} (P) + m_j^{loc}(P) \circ O_N (2^{-jN}) = m_j^{loc}(P) \circ \widetilde{\phi}_j^{loc}(P) + O_N (2^{-Nj}).
\end{equation}
Combining all the above we deduce 
\begin{eqnarray}\label{eq-ajzj}
A_j (m, P) \circ \widetilde{\phi}_j (m, P) = m_j^{loc}(P) \circ \widetilde{\phi}_j^{loc}(P) + O (2^{-jN}).
\end{eqnarray}
It completes the proof.
\end{proof}
We set
 the local operator associated to $m(P)$;
\begin{equation}
m^{loc} (P) = \sum_{j=1}^{\infty} \phi_j^{loc}(P) [m_j^{loc}(P) \circ \widetilde{\phi}_j^{loc}(P)] \phi_j^{loc}(P),
\end{equation}
Now we can deduce the following result.
\begin{prop}\label{prop-m-decom}
For $m \in L^{\infty}[0,\infty)$ we have
\begin{equation}\label{eq-m-split}
A(m,P) f = m^{loc}(P) f  + O(1) f.
\end{equation}
\end{prop}
\begin{proof}
Recall that
\begin{equation}
A(m,P) =\sum_{j=1}^{\infty}\phi_j (P) \left[ A_j (m,P) \circ \widetilde{\phi}_j (P) \right]\phi_j (P) f.
\end{equation}
Using \eqref{eq-ajb}, \eqref{eq-k-l1} and $\sum_{j=0}^{\infty}O(2^{-j}) = O(1)$ we get
\begin{equation}
\begin{split}
\sum_{j=1}^{\infty} \phi_j (P) \left[ A_j (m,P) \circ \widetilde{\phi}_j (P)\right] \phi_j (P) f &= \sum_{j=1}^{\infty} \phi_j (P) \left[  m_j^{loc}(P) \circ \widetilde{\phi}_j^{loc}(P). + O (2^{-jN})\right] \phi_j (P) f
\\
& = \sum_{j=1}^{\infty} \phi_j (P)\left[ m_j^{loc}(P) \circ \widetilde{\phi}_j^{loc}(P)\right] \phi_j (P) f + O (1) f.
\end{split}
\end{equation}
Here we note that all the $L^1$-norms of the kernels of $\phi_j (P), m_j^{loc}(P)$, and $\widetilde{\phi}_j^{loc}(P)$ with respect to the second variable are bounded unfiormly for $j \in \mathbb{N}$. 
Also, using Lemma \ref{lem-smooth-loc}, we have $\phi_j (P) = \phi_j^{loc}(P) + O (2^{-jN})$. Combining these two facts we deduce that 
\begin{equation}
\sum_{j=1}^{\infty}\phi_j (P) \left[ m_j^{loc}(P) \circ \widetilde{\phi}_j^{loc}(P)\right] \phi_j (P) f = \sum_{j=1}^{\infty}\phi_j^{loc}(P)\left[ m_j^{loc}(P)\circ \widetilde{\phi}_j^{loc}(P)\right] \phi_j^{loc} (P) f + O (1) f.
\end{equation}
It completes the proof.
\end{proof}
\section{Bound of the localized by Hardy-Littlewood maximal funtion}
We let ${H}_j$ be the kernel of the operator $m_j^{loc}(P) \circ \widetilde{\phi}_j^{loc}(P).$  Let ${\widetilde{\Phi}}_j $ be the kernel of $\int Q(s) \hat{\widetilde{\Phi}}_{j} (s) \rho(s) ds$. By Lemma \ref{kernell2} we have $\widetilde{\Phi}_j (x,y) = 2^{jn} \widetilde{\Phi}_j^* (2^j x, 2^jy)$ with $\widetilde{\Phi}_j$ satisfying
\begin{eqnarray*}
\int |D_y^{\alpha} \widetilde{\Phi}_j^* (x,y)|^2 (1+ |x-y|)^{2N} dx \leq C_N, \quad 0 \leq |\alpha| \leq 1
\end{eqnarray*}
for any $N \in \mathbb{N}$. 

Moreover we have
\begin{equation}\label{composition}
{H}_j (x,z) = \int_M {K}_j (x,y) \widetilde{\Phi}_j (y,z) dy.
\end{equation}
We have the following result.
\begin{lem}\label{lem-hj} Suppose that $m \in L^{\infty}[0,\infty)$ satisfies \eqref{smooth} for some $s>0$. Then we have ${H}_j (x,z) = 2^{jn} H_j^* (2^j x, 2^j z)$ with $H_j^*$ satisfying
\begin{eqnarray*}
\int_M |H_j^* (x,z)|^q (1+ |x-z|)^{{sq}} dz \leq C_q ([m]_s),
\end{eqnarray*}
for each $q \geq 2$. 
\end{lem}

\begin{proof}
We write \eqref{composition} as 
\begin{equation*}
\begin{split}
2^{jn} H_j^* (2^j x, 2^j z) &= \int_M 2^{jn} K_j^* (2^j x, 2^j y) 2^{jn} \widetilde{\Phi}_j^* (2^j y, 2^j z) dy
\\
&= \int_M 2^{jn} K_j^* (2^j x, y) \widetilde{\Phi}_j^* (y, 2^j z) dy.
\end{split}
\end{equation*}
Thus it holds that
\begin{eqnarray*}
H_j^* (x,z) = \int_M K_j^* (x,y) \widetilde{\Phi}_j^* (y,z) dy.
\end{eqnarray*}
Using Lemma \ref{kernell2} and H\"older's inequality, we have
\begin{equation*}
\begin{split}
(1 + |x-z|)^s |H_{j}^* (x,y)| &= ( 1+ |x-z|)^s \int_M K_{j}^* (x,y) \widetilde{\Phi}_j^* (y,z) dy
\\
&\leq \int_M K_{j}^* (x,y) ( 1+ |x-y|)^s \cdot \widetilde{\Phi}_j^* (y,z) (1+|y-z|)^s dy
\\
&\leq  \left(\int_M |K_{j}^* (x,y)|^2 ( 1 + |x-y|)^{2s} dy\right)^{1/2} \cdot \left(\int_M |\widetilde{\Phi}_j^* (y,z)|^2 (1+|y-z|)^{2s} dy \right)^{1/2}
\end{split}
\end{equation*}
It completes the proof.
\end{proof}

From Theorem \ref{thm-para} we see that the above two operators are both local operators, i.e., their kernels have supports on near the diagonal set in $M \times M$. Therefore, the kernel of the operator $m_j^{loc}(P) \circ \phi_j^{loc}(P)$ has also support near the diagonal.
\begin{lem}\label{lem-lj}
Let $\psi \in C^{\infty}(1/2,1)$ and set $\psi_j (\cdot):= \psi(\cdot/2^j)$ for $j \in \mathbb{N}$. Let $\Psi_j$ be the kernel of $\psi_j^{loc}$. Then we have
\begin{equation} 
\int {\Psi}_j (x,y) dx = O_N (2^{-jN}).
\end{equation}
\end{lem}
\begin{proof}
Since $j \geq 1$ we have that $[{\psi}_{j} (P) 1 ](x) = 0$ for all $x \in M$. Recall that ${\psi}_{j}(P)$ equals to
\begin{equation}
\begin{split}
{\psi}_j (P) &= \int e^{itP} \hat{{\psi}}_j (s) \rho (s) ds + \int e^{it P} \hat{{\psi}}_j (s) [ 1- \rho (s)] ds.
\\
&= \int [Q(s) + R(s)] \hat{{\psi}}_j (s) \rho (s) ds + \int e^{it P} \hat{{\psi}}_j (s) [1-\rho(s)] ds.
\end{split}
\end{equation}
Thus we have
\begin{equation}\label{eq-q-r-zeta}
\left[ \int [Q(s)]\hat{{\psi}}_j (s) \rho (s) dx\right] 1 (x) = - \left[\int R(s) \hat{{\psi}}_j (s) \rho (s) ds\right] 1 (x) - \left[ \int e^{itP} \hat{{\psi}}_j (s) [1-\rho(s)]dx\right] 1 (x).
\end{equation}
Observing that $R(s) \rho (s)$ is a smooth function and ${\psi}_j (s)$ is supported on $[2^{j-1}, 2^{j+1}]$ we deduce 
\begin{equation*}
\int R(s) \widetilde{\phi}_j (s) \rho(s) ds = \int \left[ R(\cdot) \rho (\cdot)\right]^{\wedge} (s) {\psi}_j (s) dx = O(2^{-jN}).
\end{equation*}
Next, we may apply Lemma \ref{lem-r} for ${\psi}$ with any $s>0$ since  ${\psi}$ is smooth. Then we have
\begin{equation}
 \int e^{itP} \hat{\psi}_j (s) [1-\rho(s)] ds= O_N (2^{-jN}).
\end{equation}
Injecting the above two estimates into \eqref{eq-q-r-zeta} we get
\begin{equation}
\left[ \int Q(s) \hat{\psi}_j (s) \rho(s) ds\right] 1 (x) = O_N (2^{-jN}).
\end{equation}
Because we have the relation $\int {\Psi}_j (x,y) dy = \int  \left[ Q(s) \hat{{\psi}}_j (s) \rho(s) dx\right] 1 (x)$, the above bound proves the lemma.
\end{proof}
\begin{cor}\label{lem-hj-2} Suppose that $m \in L^{\infty}[0,\infty)$ satisfies the condition \eqref{smooth} for some $s>0$. Then we have ${H}_j (x,z)$
\begin{eqnarray*} 
\int H_j (x,z) dz = O_N (2^{-jN})
\end{eqnarray*}
for any $N \in \mathbb{N}$.
\end{cor}

\begin{proof}
Let $K_j$ be the kernel of $\int Q(t) \hat{m}_j (t) \rho(t) dt$. By \eqref{eq-k2} and H\"older's inequality we have
\begin{equation}
\int_M \left| K_j (x,y)\right| dx \leq C.
\end{equation}
Using \eqref{composition} and Lemma \ref{lem-lj}, we may deduce that
\begin{equation*}
\begin{split}
\left|\int H_j (x,z) dz\right| &= \left|\int \left[\int \widetilde{\Phi}_j (y,z) dz\right] K_j (x,y) dy\right|
\\
&\leq \int O(2^{-Nj})|K_j (x,y)| dy
\\
&= O (2^{-Nj}).
\end{split}
\end{equation*}
It completes the proof.
\end{proof}

Up to now, we have localized the kernel by splitting it into a major part $H_j$ and a remainder part. At this stage we concentrate on the major term ${H}_j$. 
\begin{lem}Assume that $s > \frac{n}{r}$. We have
\begin{eqnarray*}
|{H}_j * f (x) | \lesssim M_r f(x) \cdot \| m_j \|_{L_2^{\alpha}}, 
\end{eqnarray*}
\end{lem}

\begin{proof}Let us take $q >2$ such that $\frac{1}{q}+\frac{1}{r} =1$. By Lemma \ref{kernell2} we have $H_j^*$ such that ${H}_j (x,z) = 2^{jn} H_j^* (2^j x, 2^j z)$ and 
\begin{equation}\label{partition}
\int_G |x-y|^{\alpha q} |H_j^* (x,y)|^{q} dy \lesssim \| \widetilde{m}_k\|_{H^{s}}^{q}, \quad \textrm{for all} ~0\leq \alpha < s.
\end{equation}  
Set $H_{k,l}^* (x,y) = H_k^{*} (x,y) \cdot 1_{\{2^{l-1} \leq |x-y| < 2^l\}}$ for $l \in \mathbb{N}$ and ${H}_{k,0}^* (x,y) = {H}_k^* (x,y) \cdot 1_{\{ |x-y| < 1\}}$. Then we deduce from \eqref{partition} that 
\begin{eqnarray}\label{alpha}
\sup_{l\geq 0} 2^{l \alpha q} \int |{H}_{k,l}^* (x,y)|^{q} dy \lesssim \| \widetilde{m}_k\|_{H^{s}}^{q} \quad \textrm{for} \quad 0 \leq \alpha < s.
\end{eqnarray} 
Since $\frac{n}{r} <s$ we can take an $\epsilon >0$ such that $\alpha :=\frac{n}{r}+\epsilon < s$. By a direct calculation we have 
\begin{equation}
\begin{split}
|m_k (L) f (x)| &=  \left| \int_G 2^{nk} H_k^* (2^{k}x, 2^{k} y) f(y) dy\right|
\\
 &=\left| \sum_{l=0}^{\infty} \int_G 2^{nk} {H}_{k,l}^* (2^{k}x, 2^{k} y) f (y) dy \right|
\\
&\leq  \sum_{l=0}^{\infty} \left( \int_G 2^{nk} |H^{*}_{k,l}(2^{k}x, 2^{k} y)|^{q} dy\right)^{1/q} \left(2^{nk}\int_{|x-y| \leq 2^{l -k}} |f(y)|^{r} dy\right)^{1/r}
\\
&\leq \sum_{l=0}^{\infty} 2^{(ln/r)l}( M (|f|^r)(x))^{1/r} \left( \int_G |{H}^*_{k,l}(y)|^{q} dy \right)^{1/q}.  
\end{split}
\end{equation}
Now we apply \eqref{alpha} to get
\begin{equation}
\begin{split}
|m_k (L) f (x)|& \leq   \| \widetilde{m}_k \|_{H^s} ~\sum_{l=0}^{\infty} 2^{ln/r} 2^{-l\alpha } (M (|f|^r ) (x))^{1/r} = \|\widetilde{m}_k\|_{H^{s}} \sum_{l=0}^{\infty} 2^{-{l \epsilon}} (M(|f|^r)(x))^{1/r}
\\
&\leq \| \widetilde{m}_k\|_{H^{s}} (M(|f|^r)(x))^{1/r}.
\end{split}
\end{equation}
It proves the lemma.
\end{proof}

Let $\Phi_J$ be the kernel of $\phi_j^{loc}$ and set 
\begin{equation}
\Phi_j^{*} (x,y) = 2^{-nj} \Phi_j (2^{-j}x, 2^{-j}y).
\end{equation}
Then we have
\begin{equation}
\sup_{j \geq 1} \int \left|\Phi_j^{*} (x,y)\right|^2 (1+|x-y|)^{2s} dx \leq C.
\end{equation}
 
\section{Martingale operators and their interaction with the multipliers}
For $k \in \mathbb{Z}$ we consider the set of dyadic cubes $[ 2^k a_1, 2^k (a_1 +1))\times \cdots \times [2^k a_n, 2^k (a_n +1))$ for each $a=(a_1,\cdots,a_n) \in \mathbb{Z}^n$. 
The expectation operator $\mathbb{E}_k$ is defined by
\begin{eqnarray*}
\mathbb{E}_k f (x) = \mu(Q_{\alpha}^{k})^{-1} \int_{Q_{\alpha}^{k}} f d\mu \quad \textrm{for} ~ x \in Q_{\alpha}^{k}.
\end{eqnarray*}
Then we define the martingale by $ \mathbb{D}_k f (x) = \mathbb{E}_{k+1} f (x) - \mathbb{E}_k f (x).$ We also define the following square function 
\begin{eqnarray*}
S (f) = \left(\sum_{k \geq 0} | \mathbb{D}_k f(x)|^2\right)^{1/2}
\end{eqnarray*}
We recall the following result on $\mathbb{E}_k$ and $S(f)$.
\begin{thm}[see {\cite[Corollary 3.1.]{CW}}] There is a constant $C_d > 0$ such that, for any $\lambda >0$, and $0 < \epsilon < \frac{1}{2}$, the following inequality holds.
\begin{equation}\label{eq-cw}
\begin{split}
\textrm{meas} ( \{ x : \sup_{k \geq 0} | \mathbb{E}_k g(x) - \mathbb{E}_{0} g(x)| > 2\lambda, S(g) < \epsilon \lambda \}) 
\\
\leq C \textrm{exp}( - \frac{C_d}{\epsilon^2}) \textrm{meas} ( \{ x : \sup_{k \geq 0} | \mathbb{E}_k g(x)| > \lambda \}) ;
\end{split}
\end{equation}
\end{thm}
 
Let us introduce the following functional
\begin{eqnarray*}
G_r (f) = ( \sum_{k \in \mathbb{N}} (\mathcal{M} (|{\phi}^{loc}_k (P) f |^r))^{2/r} )^{1/2}
\end{eqnarray*}
We have the following ineqaulity due to Fefferman-Stein \cite{FeS}.
\begin{eqnarray*}
\| G_r (f) \|_p \leq C_{p,r} \| f \|_p , \qquad 1 < r <2, r < p < \infty.
\end{eqnarray*}
In order to prove Proposition \ref{prop-main} we shall make use of inequality \eqref{cw} with $g = m^{loc}(P) f$. For this it will be requied to bound the $L^p$-norm of $S(m^{loc}(P)f$ by a constant multiplier of $\|f\|_{L^p}$.

We need the following lemma which explains the cancellation property.
\begin{lem}
$| \mathbb{E}_k ( {\phi}_j f ) (x) \leq 2^{( k - j)/q' } M_q f (x) $ if $ j > k + 10$.
\\
$| \mathbb{B}_k ({\phi}_j f ) (x) | \leq 2^{(j -k )/q'} M_q f(x)$ if $ j <  k - 10.$
\end{lem} 
\begin{proof}
For $x \in M$ we find a unique $Q_{\alpha}^{k}$ such that $x \in Q_{\alpha}^{k}$. Then we have
\begin{equation}\label{knlfx}
\begin{split}
\mathbb{E}_k (\widetilde{\phi}_j(L)  f) (x) =& \frac{1}{\mu (Q_{\alpha}^{k})} \int_{Q_{\alpha}^{k}} (\widetilde{\phi}_n(L)  f) (y) dy
\\
=&\frac{1}{\mu(Q_{\alpha}^{k})} \int_{Q_{\alpha}^{k}} \left[ \int_G 2^{Qj/2} \Phi_j^{*} (2^{j/2}y, 2^{j/2}z) f (z) dz \right] dy
\\
=& \frac{1}{\mu(Q_{\alpha}^{k})} \int_G \left[ \int_{Q_{\alpha}^{k}} 2^{Qj/2} \Phi_j^{*} (2^{j/2} y, 2^{j/2}z ) dy \right] f (z) dz.
\end{split}
\end{equation}
We set $d(k,j):= j-k$.
\

We first consider the case $d(k,j)>10$.  In \eqref{knlfx} we split the domain of the variable $z$ into the following disjoint sets: 
\begin{itemize}
\item[-] $B = \{ z : \dist(z, \partial Q^{k}_{\alpha}) \leq 2^{- [ (-\log_2 \delta) k + \frac{d(k,n)}{2}]} \}$
\item[-] $ A_1 = Q_{\alpha}^{k} \cap B^{c}$
\item[-] $ A_2 = (Q_{\alpha}^{k})^{c} \cap B^{c}$.
\end{itemize}
Then we see that $G= B \cup A_1 \cup A_2$ and we have $f = f_{A_1} + f_{A_2} + f_B := f \chi_{A_1} + f \chi_{A_2} + f \chi_{B}$. Thus we have
\begin{eqnarray*}
\mathbb{E}_k (\widetilde{\phi}_j (L) f)(x) = \mathbb{E}_k (\widetilde{\phi}_j (L) f_{A_1})(x) + 
\mathbb{E}_k (\widetilde{\phi}_j (L) f_{A_2})(x)
+\mathbb{E}_k (\widetilde{\phi}_j (L) f_{B})(x).
\end{eqnarray*}
We shall estimate the each three terms in the below.
\

\noindent$\cdot ~Estimate ~for ~f_{A_1}$. 
\\
Substituting $f$ with $f_{A_1}$ in \eqref{knlfx} we have 
\begin{equation}\label{3C}
\begin{split}
\mathbb{E}_k (\widetilde{\phi}_j (L)  f_{A_1} (x)) =&\frac{1}{\mu (Q_{\alpha}^{k})} \int_{Q^{k}_{\alpha}} \left[ \int_G 2^{nj} \Phi_j^{*}(2^{n}y, 2^{n} z) 1_{A_2} (z) f(z) dz \right] dy
\\
 =& \frac{1}{\mu(Q_{\alpha}^{k})} \int_G \left[ \int_{Q_{\alpha}^{k}} 2^{nj} \Phi_j^{*}(2^n y, 2^n z) dy \right] \chi_{A_1} (z) f (z) dz.
\end{split}
\end{equation}
Observing $\int \Phi_j^* (x,y) dy=O(2^{-j})$ we deduce that
\begin{equation}\label{3D}
\begin{split}
\left| \int_{Q_{\alpha}^{k}} 2^{nj} \Phi_j^* (2^n y, 2^n z) dy \right| =&  \left|\int_{(Q_{\alpha}^{k})^{c}} 2^{nj} \Phi_j^* (2^j y, 2^j z)dy \right|
\\
\leq & \int_{(Q_{\alpha}^{k})^{c}} 2^{nj} |\Phi_j^* (2^{j}y, 2^j z)| ~dy
\\
\leq & \int_{|y-z| \geq 2^{- [ (-\log_2 \delta) k + \frac{d(k,n)}{2}]}} 2^{nj} |\Phi_j^* (2^j y, 2^j z)| ~dw
\\
\leq & \int_{|y-z| \geq 2^{m/2}} |\Phi_j^* (y,z)| ~dy \\
\leq & \int_{|y-z|\geq 2^{m/2}} |y-z|^{-N} dy \leq 2^{-m/2 c},
\end{split}
\end{equation}
where the second inequality holds since $z \in A_1 = Q_{\alpha}^{k} \cap B^{c}$ and $y \in (Q_{\alpha}^{k})^{c}$. Combining this estimate with \eqref{3C} we obtain
\begin{equation}\label{fa1}
\begin{split}
|\mathbb{E}_k (\widetilde{\phi}_n(L)  f(x))| &\leq \frac{1}{\mu(Q_{\alpha}^{k})} \int_G 2^{-mc/2} 1_{A_1}(z) f(z) dz
\\
&\leq  2^{-mc/2} M f (x).
\end{split}
\end{equation}

\noindent $\cdot~ Estimate ~for~ f_{A_2}$.  
\

As in \eqref{3C} we have 
\begin{eqnarray}\label{ekpnl}
\mathbb{E}_k ( \widetilde{\phi}_j (L)  f_{A_2}(x)) &=& \frac{1}{\mu (Q_{\alpha}^{k})} \int_{Q^{k}_{\alpha}} \left[ \int_G 2^{nj} \Phi_j^* (2^{j}y, 2^j z) 1_{A_2} (z) f(z) dz \right] dy.
\end{eqnarray}
Observe that we have $|(y - z)| \geq 2^{- [ (-\log_2 \delta) k + \frac{d(n,k)}{2}]}$ for $ z \in A_2 =(Q_{\alpha}^{k})^{c} \cap B^{c}$ and $y \in Q_{\alpha}^{k}$. It implies $|2^{j} (y-z)| \geq 2^{j + (\log_2 \delta)k -\frac{d(n,k)}{2}} = 2^{\frac{d(n,k)}{2}}$. Then, using \eqref{kx1xn} we deduce that
\begin{eqnarray*}
\sup_{ y \in Q_{\alpha}^k} \int_{A_2} 2^{nj}\left| \Phi_j^* (2^{j}y, 2^j z)\right| dz \lesssim \int_{ |x| \geq 2^{d(n,k)/2}} (1 + |x|)^{-3N} dx \lesssim 2^{-d(n,k) N}. 
\end{eqnarray*}
By this we get 
\begin{eqnarray*}
 \int_G \sup_{ y \in Q_{\alpha}^{k}} \left| 2^{nj}  \Phi_j^* (2^{j}y, 2^j z) 1_{A_2} f (z) \right| dz  \leq M f (x) \cdot 2^{-d(n,k) N}.
\end{eqnarray*}
It enable us to estimate \eqref{ekpnl} in the following way
\begin{eqnarray}\label{fa2} 
\mathbb{E}_k (\widetilde{\phi}_j (L)  f_{K_2} (x))\lesssim \frac{2^{-d(n,k)N}}{\mu (Q_{\alpha}^k)} \int_{Q_{\alpha}^k} M f (x) dy  = M f (x) \cdot 2^{-d(n,k) N}.
\end{eqnarray}
\noindent$\cdot ~Estimate ~ for ~ f_B$.
\\
We have
\begin{equation}\label{fb}
\begin{split}
|\mathbb{E}_k (\phi_j (L) f_B )(x)| &= \frac{1}{\mu(Q_{\alpha}^{k})} \left| \int_B \left[ \int_{Q_{\alpha}^{k}} 2^{nj} \Phi_j^* (2^n y, 2^n z) dy \right] f(z) dz \right|
\\
&\leq \frac{1}{\mu (Q_{\alpha}^k)} \int_B \left[\int_{Q_{\alpha}^{k}} 2^{nj}\Phi_j^* (2^n y, 2^n z)| dy\right] f(z) dz
\\
&\leq \frac{1}{\mu(Q_{\alpha}^{k})} \int_B \left( \int_G 2^{nj} |\Phi_j^* (2^j y, 2^j z)| dy\right) |f(z)| dz
\\
&= \frac{C}{\mu(Q_{\alpha}^{k})} \int_B |f(z)| dz.
\end{split}
\end{equation}
Note that $\mu (B) \leq C \mu(Q_{\alpha}^{k}) 2^{-\frac{d(n,k)}{2}\rho}$. Using this fact,  we can estimate \eqref{fb} as follows.
\begin{equation}\label{fb2}
\begin{split}
| \mathbb{E}_k (\phi_n (L) f_B)(x)| &\leq C \frac{1}{\mu(Q_{\alpha}^{k})} \int_B |f(z) | dz 
\\
&\leq C \frac{1}{\mu(Q_{\alpha}^{k})} \mu (B)^{1/q'} \left( \int_B |f(z)|^{q} dz\right)^{1/q}
\\
&\leq C 2^{-\frac{\rho}{2q'} d(n,k)} \left( \frac{1}{\mu(Q_k^{\alpha})} \int_B |f(z)|^q dx \right)^{1/q} 
\\
&\leq C 2^{-\frac{\rho}{2q'}d(n,k) } M_q f(x).
\end{split}
\end{equation}
Now, we can combine \eqref{fa1}, \eqref{fa2} and \eqref{fb} to have 
\begin{equation*}
\begin{split}
|\mathbb{E}_k (\phi_n (L) f)| &=| \mathbb{E}_k (\phi_n (L)( f_{A_1}+ f_{A_2} + f_B))(x) |
\\
&\lesssim 2^{- d(n,k) \gamma} M_q f(x),
\end{split}
\end{equation*}
where $\gamma = \min (\frac{c}{2}, \frac{\rho}{2q'})$. It proves the lemma for the case $d(k,n)>10$. 
\

We now turn to the case $d(k,n) <10$. By the definition we have $\mathbb{D}_k (\widetilde{\phi}_j(L) f) =\mathbb{E}_{k+1} (\widetilde{\phi}_j(L)  f) -\mathbb{E}_{k} (\widetilde{\phi}_j(L)  f)$ and
\begin{equation}\label{eq-a-1}    
\begin{split} \mathbb{D}_{k}& (\widetilde{\phi}_j (L)  f) (x) \\
=& ~\frac{1}{\mu(Q_{\alpha}^{k+1})} \int_{Q_{\alpha}^{k+1}} (\widetilde{\phi}_j(L)  f) (y) dy - \frac{1}{\mu(Q_{\alpha}^{k})} \int_{Q_{\alpha}^{k}} (\widetilde{\phi}_j (L)  f)(y) dy
\\
=& ~\int_G  f(z) \left[ \frac{1}{\mu(Q_{\alpha}^{k+1})} \int_{Q_{\alpha}^{k+1}} 2^{nj} \Phi_j^{*}(2^j y, 2^j z) dy - \frac{1}{\mu(Q_{\alpha}^{k})} \int_{Q_{\alpha}^{k}} 2^{nj} \Phi_j^* (2^j y, 2^j z) dy\right] dz 
\\
=& ~\int_G  f(z) \left[ \frac{1}{\mu(Q_{\alpha}^{k+1})} \int_{Q_{\alpha}^{k+1}} 2^{nj} \left[ \Phi_j^* (2^{j}y, 2^j z) - \Phi_j^* (2^{j}x, 2^j z) \right]dy \right] dz
\\
\quad\quad&\quad\qquad\qquad- \int_G f(z) \left[ \frac{1}{\mu(Q_{\alpha}^{k})} \int_{Q_{\alpha}^{k}} 2^{nj} \left[ \Phi_j^* (2^{j}y, 2^j z)- \Phi_j^* (2^j x, 2^j z)\right] dy\right] dz 
\\
:=& ~A_1 + A_2,
\end{split}  
\end{equation}  
By the mean value theorem  there is a constant  $\beta>0$ such that
\begin{equation}\label{eq-meanvalue}
\begin{split}
&\left| \Phi_j^* ( (2^{j}(y-x) + 2^j (x), 2^j z) - \Phi_j^*  (2^{j}x, 2^j z) \right|
\\
&\quad\quad\quad\quad\quad\quad\quad\leq C \sum_{i=1}^{d} |2^{j}  (yx^{-1})|^{d_j} \sup_{|w|\leq | 2^{j} (y-x)|} \left| X_i \Phi_j^* (w+2^{j}x, 2^j z) \right|
\end{split}
\end{equation}
For $x, y \in Q_{\alpha}^{k}$ we have   $|(yx^{-1})| \leq \delta^k$, and so $|2^{n/2} (y x^{-1})|\leq 2^{n/2} 2^{(\log_2 \delta)k}\leq 2^{-10}$ by the assumption. Using this we deduce that
\begin{equation}\label{eq-a-2}
\begin{split}
&\quad \sum_{i=1}^{d} |2^{j}(x-y)|^{d_i} \sup_{|w|\leq |2^{j}(y-x)|} \left| X_i \Phi_j^{*} (w + 2^{j}x, 2^j z) \right|
\\
&\quad\quad\quad\quad\quad\quad\quad\leq C \sum_{i=1}^{d} (2^{j} \delta^{k})^{d_i} \sup_{|w|\leq  2^{-10}} \left( 1 +   |(w + 2^{j} x)- 2^j z|\right)^{-N}
\\
&\quad\quad\quad\quad\quad\quad\quad\leq C  (2^{j} \delta^{k}) \left( 1 +  | 2^{j}(x - z)|\right)^{-N}.
\end{split}
\end{equation}
Combining \eqref{eq-a-1}, \eqref{eq-meanvalue} and \eqref{eq-a-2} we get
\begin{equation}
\begin{split}
|A_1 | &\leq C (2^{j} \delta^{k}) \int_G  2^{nj} \left(1 + |2^{j} (x-z)|\right)^{-N} f(z) dz
\\
&\leq C (2^{j} \delta^{k})  M f(x).
\end{split}
\end{equation}
The same argument shows that $|A_2| \lesssim (2^{j} \delta^{k}) M f(x)$. Therefore we have shown that
\begin{eqnarray*}
|\mathbb{D}_k (\widetilde{\phi}_j (L) f ) (x) | \leq C  (2^{j} \delta^k) M f (x).
\end{eqnarray*}
It finish the proof of the lemma.
\end{proof}

\begin{prop}\label{prop} Suppose that $m \in L^{\infty}[0,\infty)$ satisfies \eqref{smooth} for $s>0$. Then, for $ x \in G$, $S (m^{loc}(P)f)(x) \leq A_r \| m \|_{L_2^{s}} G_r (f) (x)$ for each $s> \frac{d}{r}$.
\end{prop}
\begin{proof}
We have
\begin{equation*}
\begin{split}
|\mathbb{B}_k (Tf)| &= | \sum_{j \in \mathbb{Z}} \mathbb{B}_k (\bar{\phi}_j \widetilde{H}_j (x,z) \widetilde{\phi}_j (P) f ) |
\\
& \leq \sum_{j \in \mathbb{Z}} 2^{- | k  -j |} M^r ( \widetilde{\phi}_j f).
\end{split}
\end{equation*}
Using the Cauchy-Schwartz inequality we have
\begin{eqnarray*}
|\mathbb{B}_k (Tf)|^2 \leq  ( \sum_{j \in \mathbb{Z}} 2^{- |k - j |}) \sum_{j \in \mathbb{Z}} 2^{- |k -j| } (M_r (\widetilde{\phi}_j f ))^2.
\end{eqnarray*}
Summing up this, we deduce that
\begin{eqnarray*}
S(Tf) (x) = (\sum_{k=1}^{\infty} |\mathbb{B}_k (Tf)|^2)^{1/2} \leq C ( \sum_{n \in \mathbb{Z}} |M_q (\widetilde{\phi}_j f )|^2)^{1/2}.
\end{eqnarray*}
It completes the proof.
\end{proof}

\section{Proof of Proposition \ref{prop-main}}
We need to bound 
\begin{eqnarray*}
\left\| \sup_{1 \leq i \leq N } |T_i f|\right\|_p = \left( p 4^p \int^{\infty}_0 \lambda^{p-1} \textrm{meas} ( \{ x : \sup_{i} | T_i f(x)|  > 4\lambda\}) d\lambda\right)^{1/p}
\end{eqnarray*}
by some constant time of $ \sqrt{\log (N+1)} \| f\|_p$. By proposition \ref{prop} we have the pointwise bound
\begin{eqnarray}\label{SG}
S( T_i f ) \leq A_r B G_r (f).
\end{eqnarray}
We bound the level set as
\begin{eqnarray*}
\{ x : \sup_{1 \leq i \leq N} |T_i f (x)| > 4 \lambda \} \subset E_{\lambda,1} \cup E_{\lambda,2} \cup E_{\lambda,3},
\end{eqnarray*}
where  
\begin{eqnarray*}
\epsilon_N : = ( \frac{c_d} {10 \log (N+1)})^{1/2}
\end{eqnarray*}
and
\begin{equation*}
\begin{split}
E_{\lambda, 1} &= \{ x : \sup_{1 \leq i \leq N} | T_i f (x) - \mathbb{E}_{-N} T_i f (x)| > 2\lambda, G_r (f) (x) \leq \frac{\varepsilon_N \lambda}{A_r B}\},
\\
E_{\lambda, 2} &= \{ x : G_r (f)(x) > \frac{\varepsilon_N \lambda}{A_r B}\},
\\
E_{\lambda, 3} &= \{ x : \sup_{1 \leq i \leq N} | \mathbb{E}_0 T_i f (x) > 2\lambda \}.
\end{split}
\end{equation*}
By \eqref{SG}, 
\begin{eqnarray*}
E_{\lambda,1} \subset \bigcup_{i=1}^{N} \{ x : |T_i f(x)| > 2\lambda, S (T_i f ) \leq \varepsilon_N \lambda \}
\end{eqnarray*}
and  we have
\begin{equation*}
\begin{split}
\textrm{meas}(E_{\lambda,1}) &\leq \sum_{i=1}^{N} \textrm{meas} ( \{ x : |T_i f (x) - \mathbb{E}_{-N} T_i f (x)| > 2\lambda, S(T_i f) \leq \varepsilon_N \lambda\})
\\
&\leq  \sum_{i=1}^{N} C \exp ( - \frac{c_d}{\varepsilon_N^2}) \textrm{meas} ( \{ x : \sup_{k} | \mathbb{E}_k (T_i f)| > \lambda \}).
\end{split}
\end{equation*}
Therefore
\begin{equation*}
\begin{split}
\left( p \int_0^{\infty} \lambda^{p-1} \textrm{meas} (E_{\lambda, 1}) d\lambda\right)^{1/p}& \leq C \left( \sum_{i=1}^{N} \exp (-\frac{c_d}{\varepsilon_N^2} )\left\| \sup_{k} |\mathbb{E}_k (T_i f)\right\|_p^p \right)^{1/p}
\\
&\leq C \left( \sum_{i=1}^{N} \exp (-\frac{c_d}{\varepsilon_N^2}) \left\| T_i f\right\|_p^p\right)^{1/p}
\\
&\leq C B \left( N \exp (-\frac{c_d}{\varepsilon_N^2})\right)^{1/p} \|f \|_p
\\
&\leq C B\|f\|_p.
\end{split}
\end{equation*}
Using a change of variables we have 
\begin{equation*}
\begin{split}
\left( p \int_0^{\infty} \lambda^{p-1} \textrm{meas}(E_{\lambda,2}) d\lambda\right)^{1/p} &= \frac{A_r B}{\varepsilon_N} \| G_r (f)\|_p
\\
&\leq C B \sqrt{\log (N+1)} \| f\|_p.
\end{split}
\end{equation*}
Finally, by  the Fefferman-Stein inequality we have
\begin{eqnarray*}
\textrm{meas}(E_{\lambda,3} ) \leq \sum_{i=1}^{N} \textrm{meas} ( \{ x : |\mathbb{E}_{-N} T_i f(x)| > 2 \lambda\})
\end{eqnarray*}
and thus
\begin{equation*}
\begin{split}
\left( p \int_0^{\infty} \lambda^{p-1} \textrm{meas}(E_{\lambda,3}) d\lambda\right)^{1/p} &= 2 \left\| \sup_{i=1,\dots,N} |\mathbb{E}_{-N} (T_i f)| \right\|_p
\\
&\leq  \sup_{i=1,\cdots,N} \| T_i f\|_p
\\
&\leq   \| f\|_p.
\end{split}
\end{equation*}
The above estimates completes the proof.


\begin{thebibliography}{20}

\bibitem{CW} S.Y.A Chang, M. Wilson, T. Wolff, Some weighted norm inequalities concerning the Schr\"odinger operator, Comment. Math. Helv. 60 (1985) 217-246.

\bibitem{choi} W. Choi,  Maximal multiplier on Stratified groups arXiv:1206.2817
\bibitem{C1} M. Christ, $L^p$   bounds for spectral multipliers on nilpotent groups. Trans. Amer. Math. Soc. 328 (1991), no. 1, 73–81.

\bibitem{C2} \bysame,  Lectures on singular integral operators. CBMS Regional Conference Series in Mathematics, 77 (1990). 

\bibitem{CG} M. Christ, L. Grafakos, P. Honzik, A. Seeger, Maximal functions associated with multipliers of Mikhlin-H\"ormander type, Math. Z. 249 (2005) 223-240.

\bibitem{FeS} C. Fefferman, E.M. Stein, Some maximal inequalities, Amer. J. Math. 93 (1971) 107-115.

\bibitem{FS} G.B. Folland, E.M. Stein,  Hardy spaces on homogeneous groups. Mathematical Notes, 28. Princeton University Press.; University of Tokyo Press, 1982.

\bibitem{GHS} L. Grafakos, P. Honzik, A. Seeger, On maximal functions for Mikhlin-H\"ormander multipliers, Adv in Math. 204 (2006) 363-378.


\bibitem{H} L. H\"ormander, Estimates for translation invariant operators in $L^p$ spaces. Acta Math. 104 (1960) 93-139.


\bibitem{MM} G. Mauceri, S. Meda, Vector-valued multipliers on stratified groups. Rev. Mat. Iberoamericana 6 (1990), no. 3-4, 141–154.

\bibitem{ss} A. Seeger, C.D. Sogge, On the boundedness of functions of (pseudo-) differential operators on compact manifolds. Duke Math. J. 59 (1989), no. 3, 709–736
\bibitem{sogge-annals} C.D. Sogge, On the convergence of Riesz means on compact manifolds, Ann. of Math. 126 (1987), 439-447.

\bibitem{sogge}  \bysame, Fourier integrals in classical analysis. Cambridge Tracts in Mathematics, 105. Cambridge University Press, Cambridge, 1993

\bibitem{taylor} M. Tayloer, Pseudo-differential Operators, Princeton Univ. Press, Princeton N.J., 1981.

\end{thebibliography}
\end{document}